\documentclass[11pt, twoside]{article} 
\usepackage{relsize,amsmath,amsthm,mathtools,bm,tensor, graphicx} 
\usepackage{microtype,lmodern} 
\usepackage[utf8]{inputenc} 
\usepackage[T1]{fontenc} 
\usepackage{enumerate,comment,braket,xspace,tikz-cd,csquotes} 
\usepackage[centering,vscale=0.7,hscale=0.7]{geometry}
\usepackage{xr}
\usepackage[hidelinks,linktoc=all]{hyperref}

\usepackage[garamond,cal=cmcal]{mathdesign}
\usepackage{eucal}

\usepackage[capitalize]{cleveref}
\usepackage{fancyhdr}
\usepackage{lscape}
\usepackage{adjustbox}
\usepackage{faktor}
\numberwithin{equation}{section}
\usepackage{appendix}
\usepackage{stmaryrd}
\usepackage{subfiles}
\theoremstyle{plain}

\newtheorem{thm}{Theorem}[section]
\newtheorem*{thm*}{Theorem}
\newtheorem{lem}[thm]{Lemma}
\newtheorem*{lem*}{Lemma}
\newtheorem*{prop*}{Proposition}
\newtheorem{prop}[thm]{Proposition}

\newtheorem{cor}[thm]{Corollary}

\theoremstyle{definition}
\newtheorem{defin}[thm]{Definition}

\newtheorem{rem}[thm]{Remark}
\newtheorem{construction}[thm]{Construction}

\makeatletter

\renewcommand\subsubsection{\@startsection {subsubsection}{1}{\z@}%
	{-3.5ex \@plus -1ex \@minus -.2ex}%
	{-1em}%
	{\normalfont\large\underline}}

\makeatother

\newcommand{\B}{\textup{B}}
\newcommand{\Z}{\mathbb Z}
\newcommand{\tK}{\textup{K}}
\newcommand{\Sp}{\mathrm{Sp}}
\newcommand{\Equiv}{\mathcal{E}\textup{quiv}}
\newcommand{\Map}{\textup{Map}}
\newcommand{\act}{\textup{act}}

\newcommand{\cc}{\mathcal C}
\newcommand{\cD}{\mathcal D}

\newcommand{\cS}{\mathcal S}
\newcommand{\Gm}{\mathbb G_{\mathrm{m}}}
\newcommand{\id}{\textup{id}}
\newcommand{\QCoh}{\textup{QCoh}}
\newcommand{\BGm}{\textup{B}\Gm}
\renewcommand{\Pr}{\mathcal P\textup{r}}

\newcommand{\Triv}{\mathcal{T}\textup{riv}}
\newcommand{\LMod}{\textup{LMod}}

\newcommand{\tH}{\textup{H}}
\newcommand{\cG}{\mathcal G}

\newcommand{\Vect}{\textup{Vect}}
\newcommand{\Cat}{\textup{Cat}}
\newcommand{\cE}{\mathcal{E}}
\newcommand{\cO}{{\mathcal O}}
\newcommand{\cF}{{\mathcal F}}
\newcommand{\cEnd}{{\mathcal{E}\textup{nd}}}
\newcommand{\cM}{\mathcal M}
\newcommand{\cH}{\mathcal H}
\newcommand{\cL}{\mathcal L}
\newcommand{\cQCoh}{\mathcal{QC}\textup{oh}}

\newcommand{\Sch}{\textup{Sch}}
\newcommand{\Br}{\mathsf{Br}}
\newcommand{\cBr}{\mathcal{B}\textup{r}}
\newcommand{\Bre}{\mathsf{Br}^\dagger}
\newcommand{\cBre}{\mathcal{B}\textup{r}^\dagger}

\newcommand{\Ger}{\textup{Ger}}

\newcommand{\cPic}{{\mathcal P\textup{ic}}}
\newcommand{\Pic}{\textup{Pic}}
\newcommand{\cI}{\mathcal I}
\newcommand{\pr}{\textup{pr}}

\newcommand{\Mod}{\textup{Mod}}

\newcommand{\ett}{{\textup{\'et}}}

\newcommand{\Stk}{{\mathcal S\textup{tk}}}
\newcommand{\GL}{{\textup{GL}}}
\newcommand{\PGL}{{\textup{PGL}}}
\newcommand{\Lin}{\mathcal L\textup{in}\mathcal{C}\textup{at}^{\textup{St}}}
\newcommand{\Linpre}{{\mathcal L\textup{in}\mathcal{C}\textup{at}^{\textup{PSt}}}}
\newcommand{\QStk}{\mathcal Q\textup{Stk}^\textup{St}}
\newcommand{\QStkpre}{\mathcal Q\textup{Stk}^\textup{PSt}}
\newcommand{\Deraz}{\mathcal D\textup{eraz}}

\newcommand{\AbGer}{\textup{AbGer}}
\newcommand{\Band}{\textup{Band}}
\newcommand{\AbGr}{\textup{AbGr}}
\newcommand{\Fin}{\textup{Fin}}
\newcommand{\Groth}{\cG\textup{roth}}
\newcommand{\Spec}{\textup{Spec}}
\newcommand{\cY}{\mathcal Y}

\begin{document}
	\setlength{\headheight}{14.49998pt}
	\title{The derived Brauer map via twisted sheaves}
	
	\author{Guglielmo Nocera\footnote{LAGA, Universit\'e Paris 13, Av. J.-B. Cl\'ement 99, 93430 Villetaneuse. \texttt{guglielmo.nocera-at-gmail.com}}\hspace{0.2cm} and Michele Pernice\footnote{KTH, Royal Institute of Technology, Brinellvägen 8, 114 28 Stockholm. email: \texttt{mpernice-at-kth.se}}}

	\date{\today}

	\maketitle

	\abstract{Let $X$ be a quasicompact quasiseparated scheme. The collection of derived Azumaya algebras in the sense of To\"en forms a group, which contains the classical Brauer group of $X$ and which we call $\Bre(X)$ following Lurie. To\"en introduced a map $\phi:\Bre(X)\to \tH^2_\ett(X,\Gm)$ which extends the classical Brauer map, but instead of being injective, it is surjective. In this paper we study the restriction of $\phi$ to a subgroup $\Br(X)\subset\Bre(X)$, which we call the \textit{derived Brauer group}, on which $\phi$ becomes an isomorphism $\Br(X)\simeq \tH^2_\ett(X,\Gm)$. This map may be interpreted as a derived version of the classical Brauer map which offers a way to ``fill the gap'' between the classical Brauer group and the cohomogical Brauer group. The group $\Br(X)$ was introduced by Lurie by making use of the theory of prestable $\infty$-categories. There, the mentioned isomorphism of abelian groups was deduced from an equivalence of $\infty$-categories between the \textit{Brauer space} of invertible presentable prestable $\cO_X$-linear categories, and the space $\Map(X,\tK(\Gm,2))$. We offer an alternative proof of this equivalence of $\infty$-categories, characterizing the functor from the left to the right via gerbes of connective trivializations, and its inverse via connective twisted sheaves. We also prove that this equivalence carries a symmetric monoidal structure, thus proving a conjecture of Binda an Porta.}
	
	
	\tableofcontents

	\subsection*{Acknowledgments}We would like to thank Mauro Porta who suggested the topic to us and provided many enlightening comments. Our thanks also go to Federico Barbacovi, Federico Binda, Jacob Lurie, David Rydh and Angelo Vistoli for fruitful discussions carried out with them on this matter. We also thank the referee, whose comments helped in making the paper more readable.
 
 During the months of April and May 2022, Guglielmo Nocera was hosted at MIT by Roman Bezrukavnikov.

	\section{Introduction}\label{section-history}
	Throughout the whole work, $X$ will be a quasicompact quasiseparated scheme over some field $k$ of arbitrary characteristic. The categories of sheaves and the various cohomology groups will always be understood with respect to the \'etale topology.
	
	In 1966, Grothendieck \cite{Grothendieck-Brauer-I} introduced the notion of \textit{Azumaya algebra} over $X$: this is an \'etale sheaf of algebras which is locally of the form $\cEnd(\cE)$, the sheaf of endomorphisms of a vector bundle $\cE$ over $X$. This is indeed a notion of ``local triviality'' in the sense of \textit{Morita theory}: two sheaves of algebras $A,A'$ are said to be \textit{Morita equivalent} if the category $$\LMod_{A}=\{\cF\textup{ quasicoherent sheaf over }X\textup{ together with a left action of }A\}$$ and its counterpart $\LMod_{A'}$ are (abstractly) equivalent: for instance, one can prove that, for any vector bundle $\cE$ over $X$, $\LMod_{\cEnd(\cE)}$ is equivalent to $\LMod_{\cO_X}=\QCoh(X)$ via the functor $M \mapsto \cE^{\vee}\otimes_{\cEnd(\cE)}M$.

	The classical \textit{Brauer group} $\textup{Br}_{\textup{Az}}(X)$ of $X$ is the set of Azumaya algebras up to Morita equivalence, with the operation of tensor product of sheaves of algebras. Grothendieck showed that this group injects into $\tH^2(X,\Gm)$ by using cohomological arguments: essentially, he used the fact that a vector bundle corresponds to a $\GL_n$-torsor for some $n$, and that there exists a short exact sequence of groups
	$$1\to \Gm\to \GL_n\to \PGL_n\to 1.$$
	The image of $\textup{Br}_{\textup{Az}}(X)\hookrightarrow\tH^2(X,\Gm)$ is contained in the torsion subgroup of $\tH^2(X,\Gm)$, which is often called the \textit{cohomological Brauer group} of $X$.\footnote{Other authors, however, use this name for the whole $\tH^2(X,\Gm)$.}
	
	One of the developments of Grothendieck's approach to the study of the Brauer group is due to Bertrand To\"en and its use of derived algebraic geometry in \cite{Toen-Azumaya}. There, he introduced the notion of \textit{derived Azumaya algebra} as a natural generalization of the usual notion of Azumaya algebra. Derived Azumaya algebras over $X$ form a dg-category $\Deraz_X$. There is a functor $\Deraz_X\to \mathbb{D}g^c(X)$, this latter being (in To\"en's notation) the dg-category of presentable stable $\cO_X$-linear dg-categories\footnote{See \cref{section-prestable} for a definition of $\cO_X$-linear categories in the $\infty$-categorical setting.} which are compactly generated, with, as morphisms, functors preserving all colimits. The functor is defined by $$A\mapsto \LMod_A=\{\textup{quasicoherent sheaves on $X$ with a left action of $A$}\}$$
 
    where all terms have now to be understood in a derived sense. One can prove that this functor sends the tensor product of sheaves of algebras to the tensor product of presentable $\cO_X$-linear dg-categories, whose unit is $\QCoh(X)$. Building on classical Morita theory, To\"en defined two derived Azumaya algebras to be Morita equivalent if the dg-categories of left modules are (abstractly) equivalent. This agrees with the fact mentioned above that $\LMod_{\cEnd(\cE)}\simeq \QCoh(X)$ for any $\cE\in \Vect(X)$.
	
	In \cite[Proposition 1.5]{Toen-Azumaya}, To\"en characterized the objects in the essential image of the functor $A\mapsto \LMod_A$ as the compactly generated presentable $\cO_X$-linear dg-categories which are \textit{invertible} with respect to the tensor product, i.e. those $M$ for which there exists another presentable compactly generated $\cO_X$-linear dg-category $M^\vee$ and equivalences $\mathbf 1\xrightarrow{\sim} M\otimes M^\vee$ and $M^\vee\otimes M\xrightarrow{\sim} \mathbf 1.$ The proof of this characterization goes roughly as follows: given a compactly generated invertible dg-category $M$, one can always suppose that $M$ is generated by some single compact generator $\cE_M$. Now, compactly generated presentable $\cO_X$-linear dg-categories satisfy an important \'etale descent property \cite[Theorem 3.7]{Toen-Azumaya}; from this and from \cite[Proposition 3.6]{Toen-Azumaya} one can deduce that the $\cEnd_M(\cE_M)$ has a natural structure of a quasicoherent sheaf of $\cO_X$-algebras $A$, and one can prove that $M\simeq \LMod_A$ as $\cO_X$-linear $\infty$-categories.

	Both Antieau-Gepner \cite{Antieau-Gepner} and Lurie \cite[Chapter 11]{SAG} resumed To\"en's work, using the language of $\infty$-categories in replacement of that of dg-categories. Lurie also generalized the notion of Azumaya algebra and Brauer group to spectral algebraic spaces, see \cite[Section 11.5.3]{SAG}\footnote{One should always keep in mind that a spectral algebraic space, although very general and derived in nature, is by definition a \textit{connective} spectral Deligne-Mumford stack, see \cite[Definition 1.6.8.1, Definition 1.4.4.2]{HA}.}. He considers the $\infty$-groupoid of compactly generated presentable $\cO_X$-linear $\infty$-categories which are invertible with respect to the Lurie tensor product $\otimes$, and calls it the \textit{extended Brauer space} $\cBre_X$. This terminology is motivated by the fact that the set $\pi_0\cBre_X$ has a natural abelian group structure, and by \cite[Corollary 2.12]{Toen-Azumaya} is isomorphic to $\tH^2_{\ett}(X,\Gm)\times \tH^1_{\ett}(X,\Z)$: in particular, it contains the cohomological Brauer group of $X$. At the categorical level, Lurie proves that there is an equivalence of $\infty$-groupoids between $\cBre(X)$ and $\Map_{\Stk_k}(X,\B^2\Gm\times \B\Z)$ ($\Stk_k$ is the $\infty$-category of stacks over the base field $k$). In particular, $\cBre(X)$ is 2-truncated.
	
	We can summarize the situation in the following chain of functors:
	\begin{equation}\label{equation-Bre}\Deraz_X[\mbox{Morita}^{-1}]^\simeq\xrightarrow{\sim} \cBre(X)\xrightarrow{\sim} \Map_{\Stk_k}(X,\B^2\Gm\times \B\Z) 
	\end{equation}
	where the left term is the maximal $\infty$-groupoid in the localization of the $\infty$-category of derived Azumaya algebras to Morita equivalences. At the level of dg-categories, this chain of equivalences is proven in \cite[Corollary 3.8]{Toen-Azumaya}. At the level of $\infty$-categories, this is the combination of \cite[Proposition 11.5.3.10]{SAG} and \cite[11.5.5.4]{SAG}
	
	Note that, while in the classical case we had an injection $\textup{Br}_{\textup{Az}}(X)\hookrightarrow \tH^2(X,\Gm)$, in the derived setting one has a surjection $\Bre(X):=\pi_0\cBre(X)\twoheadrightarrow \tH^2(X,\Gm)$. If $\tH^1(X,\Z)=0$ (e.g. when $X$ is a normal scheme), then the surjection becomes an isomorphism of abelian groups.
	
	While the first equivalence in \eqref{equation-Bre} is completely explicit in the works of To\"en and Lurie, the second one leaves a couple of questions open:
	\begin{itemize}\item since the space $\Map(X,\B^2\Gm\times \B\Z)$ is the space of pairs $(G, P)$, where $G$ is a $\Gm$-gerbe over $X$ and $P$ is a $\Z$-torsor over $X$, it is natural to ask what are the gerbe and the torsor naturally associated to an element of $\cBre(X)$ according to the above equivalence. This is not explicit in the proofs of To\"en and Lurie, which never mention the words ``gerbe'' and ``torsor'', but rather computes the homotopy group sheaves of a sheaf of spaces $\underline\cBr^\dagger_X$ over $X$ whose global sections are $\cBre(X)$.
		
		\item conversely: given a pair $(G,P)$, what is the $\infty$-category associated to it along the above equivalence? \end{itemize}
	
	The goal of the present paper is to give a partial answer to the two questions above. The reason for the word ``partial'' is that we will neglect the part of the discussion regarding torsors, postponing it to a forthcoming work, and focus only on the relationship between linear $\infty$-categories/derived Azumaya algebras and $\Gm$-gerbes.

In fact, we will work with an intermediate group $\Br(X)$, which we call the \textit{derived Brauer group} of $X$ (see \cref{defin-derived-Brauer}), sitting in a chain of injective maps
$$\textup{Br}_{\textup{Az}}(X)\hookrightarrow \Br(X)\hookrightarrow \Bre(X).$$
The derived Brauer group is always isomorphic to $\tH^2_\ett(X,\Gm)$ (\cref{cohomological-interpretation-brauer}) and the first injection is (up to that identification) exactly the classical Brauer map.

The following is our main theorem. The next two subsections will provide the necessary vocabulary to understand its statement. \cref{outline} provides a sketch of the main geometric ideas behind the proof.
\begin{thm*}[\cref{main-theorem}]
		Let $X$ be a qcqs scheme over a field $k$. Then there is a symmetric monoidal equivalence of $\infty$-groupoids\footnote{Actually, both sides are 2-truncated, see \cref{2-groupoids}.}
		$$\Phi:\cBr(X)^{\otimes}\longleftrightarrow \Ger_{\Gm}(X)^{\mathlarger \star}:\Psi$$
		where
		\begin{itemize}
			\item $\Phi(M)=\Triv_{\geq 0}(M)$
			\item $\Psi(G)=\QCoh_\id(G)_{\geq 0}$.
		\end{itemize}
	\end{thm*}
	We will define all the necessary vocabulary in \cref{Construction-section}. Informally, $\Triv_{\geq 0}(M)$ is a stack over $X$ which, \'etale-locally in $X$, parametrizes equivalences between $M$ and $\QCoh(X)_{\geq 0}$, the unit object in $\cBr(X)$; $\QCoh_\id(G)_{\geq 0}$ is the $\infty$-category of quasicoherent sheaves over $G$ which are connective, and homogeneous with respect to the identity character of $\Gm$ in an appropriate sense (\cref{remark-homogeneous}): without the connectivity assumption, these are known in the literature as \textit{twisted sheaves}.
 
	\begin{rem}\label{derived-Brauer-H2}
	    By looking at \cref{defin-derived-Brauer}, one sees that by taking the $\pi_0$ of the equivalence in \cref{main-theorem} one obtains an isomorphism of abelian groups $$\Br(X)\simeq \tH^2(X,\Gm).$$ This isomorphism also appears in \cite[Example 11.5.7.15]{SAG}, but as mentioned before, it is a consequence of the equivalence of $\infty$-categories $$\cBr(X)\simeq \Map(X,\B^2\Gm)$$ whose proof does not use the interpretation of the right-hand-side as the space of $\Gm$-gerbes over $X$.
	\end{rem}

 \begin{rem}
     \cref{external-product-theorem} and \cref{rigidification-pullback-theorem}, which constitute the two steps of the proof that $\Psi$ carries a symmetric monoidal structure, provide a proof of \cite[Conjecture 5.27]{BP}.
 \end{rem}

 \section{Construction of the correspondence}\label{Construction-section}

 In the present section, we construct the two functors $\Psi$ (\cref{functor-psi}) and $\Phi$ (\cref{defin-triv}) appearing in \cref{main-theorem}.
 
	\subsection{Reminders on gerbes and twisted sheaves}\label{section-gerbes}
 \begin{defin}Let $X$ be a quasicompact quasiseparated (qcqs) scheme over a field $k$, and $\cY\to X$ a stack. The group stack $\cI_{\cY}:= \cY \times_{\cY \times_X \cY} \cY$ over $\cY$  is called the inertia group stack of $\cY$ over $X$ (\cite[Definition 8.1.17]{Olsson}).\end{defin}
	\begin{defin}[{\cite[Definition 12.2.2 with $\boldsymbol{\mu}=\Gm\times X$]{Olsson}}]
	    Let $X$ be a quasicompact quasiseparated (qcqs) scheme over a field $k$. A \textit{$\Gm$-gerbe over $X$} is the datum of: \begin{itemize}\item a stack in groupoids $G$ with a map $\alpha:G\to X$ which is, \'etale-locally in $X$,
     \begin{itemize}\item nonempty (it has a section)
     \item connected (every two sections are isomorphic).
     \end{itemize}
     \item an isomorphism of group stacks over $X$ $$\Gm\times X\xrightarrow{\sim}\alpha_*\cI_G$$ which we call \textbf{banding}.
     \end{itemize}
     A \textit{morphism of $\Gm$-gerbes} over $X$ is a morphism of stacks whose induced morphisms at the level of inertia groups commutes with the bandings.
     \end{defin}

\begin{defin}
    Let $X$ be a qcqs scheme. We denote by $\Ger_{\Gm}(X)$ the $(2,0)$-category whose objects are $\Gm$-gerbes over $X$, whose $1$-morphisms are morphisms of $\Gm$-gerbes and whose $2$-morphisms are natural transformations of maps of stacks.
\end{defin}
\begin{rem}
    Note that the $\Ger_{\Gm}(X)$ is indeed a $(2,0)$-category, since all morphisms of $\Gm$-gerbes are invertible (\cite[Lemma 12.2.4]{Olsson}).
\end{rem}

The following theorem can be proven by using the same arguments of \cite[Theorem 12.2.8]{Olsson}, but recording the identifications used in the proof as 1- and 2-morphisms.
\begin{thm}\label{gerbes-H2}
    There is an equivalence of $(2,0)$-categories
    $$\Ger_{\Gm}(X)\simeq \Map(X,\B^2\Gm).$$
\end{thm}

We will not use this theorem in the proof of \cref{main-theorem}, but we will use it to justify some comments and comparisons with perspectives used by other Authors.

  \begin{construction}
     Let $G$ be a $\Gm$-gerbe over $X$. The derived $\infty$-category of quasicoherent sheaves on $G$ is denoted by $\QCoh(G)$ and it is a presentable stable compactly generated $\cO_X$-linear category. This follows from the combination of the following facts:
     \begin{itemize}\item it is true for $\BGm$ (\cite[Example 8.6]{Hall-Rydh} or \cite[Remark 5.8]{BP});
     \item every $\Gm$-gerbe $G$ over $X$ is \'etale-locally equivalent to $\BGm\times X$ (see \cite[\texttt{Tag 06QG}]{Stacks}); 
     \item presentable compactly generated categories satisfy \'etale descent (\cite[Theorem D.5.3.1]{SAG}.\end{itemize}
     We now recall the definition of $G$-twisted sheaves on $X$. This notion dates back to Giraud \cite{Gir} and later to Max Lieblich's thesis \cite{Lieblich}, and has been developed in the derived setting by Bergh and Schn\"urer \cite{Bergh-Schnurer} (using the language of triangulated categories) and Binda and Porta \cite{BP} (using the language of stable $\infty$-categories).

   Let now $\cF\in \QCoh(G)$. We have that $\cF$ is endowed with a canonical right action by $\cI_{G}$, called the \textit{inertial action} (see e.g. \cite[Section 3]{Bergh-Schnurer}).
	\end{construction}
	
	\begin{prop}\label{zero-homogeneous}
		Let $G$ be a $\Gm$-gerbe over a qcqs scheme $X$. The pullback functor $\QCoh(X)\to \QCoh(G)$ establishes an equivalence between $\QCoh(X)$ and the full subcategory of $\QCoh(G)$ spanned by those sheaves on which the inertial action is trivial.
	\end{prop}
	
	Note that the banding $\Gm \times G\to \cI_{G}$ induces a right action $\rho$ of $\Gm$ on any sheaf $\cF\in \QCoh(G)$, by composing the banding with the inertial action. On the other hand, for any character $\chi:\Gm\to \Gm$, $\Gm$ acts on $\cF$ on the left by scalar multiplication precomposed with $\chi$. Let us call this latter action $\sigma_\chi$.
	\begin{rem}Let $1$ be the trivial character of $\Gm$. \cref{zero-homogeneous} can be restated as: the pullback functor $\QCoh(X)\to \QCoh(G)$ induces an equivalence between $\QCoh(X)$ and the full subcategory of $\QCoh(G)$ where $\rho=\sigma_1$.
	\end{rem}
	
Let $G$ be a $\Gm$-gerbe over a qcqs scheme $X$, and $\chi$ a character of $\Gm$. We define the category of $\chi$-homogeneous sheaves over $G$, informally, as the full subcategory $\QCoh_\chi(G)$ of $\QCoh(G)$ spanned by those sheaves on which $\rho=\sigma_\chi$.
		
	\begin{rem}\label{remark-homogeneous}The above definition is a little imprecise, in that it does not specify the equivalences $\rho(\gamma,\cF)\simeq \sigma_\chi(\gamma,\cF), \gamma\in \Gm, \cF\in \QCoh(G)$. A formal definition is given in \cite[Definition 5.14]{BP}. There, the Authors define an idempotent functor $(-)_\chi: \QCoh(G)\to \QCoh(G)$, taking the ``$\chi$-homogeneous component''. This functor is t-exact (\cite[Proof of Lemma 5.17]{BP}), and comes with canonical maps $i_{\chi,\cF}:\cF_\chi\to \cF$. The category $\QCoh_\chi(G)$ is defined as the full subcategory of $\QCoh(G)$ spanned by those $\cF$ such that $i_{\chi}(\cF)$ is an equivalence. 
	To fix the notations, let us make these constructions explicit. If $X$ is a scheme and $G$ a $\Gm$-gerbe, we have the following diagram 
	$$\begin{tikzcd}
G \arrow[r, "u"] & G\times \B\Gm \arrow[r, "\act_\alpha"] \arrow[d, "p"] \arrow[ld, "q"] & G \arrow[d, "\pi"] \\
\B\Gm            & G \arrow[r, "\pi"]                                                   & X     \end{tikzcd}$$
	where $p,q$ are the projections, $u$ is the atlas of the trivial gerbe and $\act_{\alpha}$ is the morphism induced by the banding $\alpha$ of $G$. See \cite[Section 5]{BP} for more specific description of these maps. We denote by $L_{\chi}$ the line bundle over $\B\Gm$ associated to the character $\chi$, while $\cL_{\chi}:=q^*L_{\chi}$. Given $\cF \in \QCoh(G)$ and $\chi$ a character of $\Gm$, we define (following \cite[Construction 5.17 and proof of Lemma 5.20]{BP})
	$$ (\cF)_{\chi}:= p_*(\act_{\alpha}^*(\cF) \otimes \cL_{\chi}^{\vee}).$$
Intuitively, $\act_\alpha^* \cF$ gains a decomposition into $\Gm$-equivariant summands; tensoring with $\cL_\chi^\vee$ shifts the weights of this decomposition by $\chi^{-1}$, making the $\chi$-equivariant summand into the $1$-equivariant (i.e. invariant) summand. Pushing forward along $p$ selects this invariant component as a quasicoherent sheaf over $G$.

  Note that $(\cF)_\chi$ can be also seen as (notations as in the diagram above) $$u^*(p^*p_*(\act_\alpha^*\cF\otimes \cL_{\chi}^{\vee})\otimes \cL_\chi).$$ 
  
  By pulling back along $u$ the counit of the adjunction $p^*\dashv p_*$ we obtain a canonical morphism $i_{\chi}(\cF):(\cF)_{\chi} \rightarrow \cF$.
	
\cite{BP} also prove that there is a decomposition $$\QCoh(G)\simeq\prod_{\chi\in \textup{Hom}(\Gm,\Gm)}\QCoh_\chi(G)$$
    building on the fact that $\cF\simeq \bigoplus_{\chi\in \textup{Hom}(\Gm,\Gm)}\cF_\chi$. The same result was previously obtained by Lieblich in the setting of abelian categories and by Bergh-Schn\"urer in the setting of triangulated categories.\end{rem}
	
	\begin{defin}\label{twisted-sheaves}
		In the case when $\chi$ is the identity character $\id:\Gm\to \Gm$, $\QCoh_\chi(G)$ is usually called the \textit{category of $G$-twisted sheaves on $X$}.
	\end{defin}
	
	The relationship between categories of twisted sheaves and Azumaya algebras has been intensively studied, see \cite{DeJong}, \cite{DeJong-Gabber}, \cite{Lieblich}, \cite{Hall-Rydh}, \cite{Bergh-Schnurer}, \cite{BP}. Given a $\Gm$-gerbe $G$ over $X$, its category of twisted sheaves $\QCoh_\id(G)$ admits a compact generator, whose algebra of endomorphisms is a derived Azumaya algebra $A_G$. In contrast, if we restrict ourselves to the setting of abelian categories and consider \textit{abelian} categories of twisted sheaves, this reconstruction mechanism does not work anymore. This is one of the reasons of the success of To\"en's derived approach.

 \begin{rem}\label{bi-homogeneous}
     Let $G,G'$ be $\Gm$-gerbes over $X$, and $\chi,\chi'$ two characters of $\Gm$. A straightforward generalization of the above definitions allows to define the sub-$\infty$-category $$\QCoh_{(\chi,\chi')}(G\times_XG')\subset \QCoh(G\times_XG')$$ spanned by ``$(\chi,\chi')$-homogeneous sheaves'': it suffices to replace tha line bundle $L_\chi$ by the external product $L_\chi\boxtimes L_{\chi'}$.
 \end{rem}

\begin{rem}\label{equalizer}
	 Let $C$ be a presentable $\infty$-category, $p:C \rightarrow C$ be an endofunctor and $\eta: p \Rightarrow \id_C$ be a natural transformation. We have a diagram 
 $$\begin{tikzcd}
    C \arrow[bend left]{r}[name=LUU]{p}
    \arrow[bend right]{r}[name=LDD,below]{\id_C}
    \arrow[Rightarrow,to path=(LUU) -- (LDD)\tikztonodes]{r}{\eta}
    & 
    C
  \end{tikzcd}$$ of $\infty$-categories. Let $C^0$ be the subcategory of $C$ spanned by the elements $X$ of $C$ such that $\eta(X)$ is an equivalence. 
Now consider $(C_1,p_1,\eta_1)$ and $(C_2,p_2,\eta_2)$ triples as the above one, and let $\rho:C_1 \rightarrow C_2$ be a functor and $\alpha: p_2\circ \rho \Rightarrow \rho \circ p_1$ be an equivalence of functors. If $(\id_{\rho} * \eta_1) \circ \alpha = (\eta_2 \circ \id_{\rho})$ then $(\rho,\alpha)$ is a morphism between the two diagrams and therefore there exists a unique morphism $C^0_1 \rightarrow C^0_2$ compatible with all the data.

If $G$ is a $\Gm$-gerbe and $\chi$ is a character of $\Gm$, the endofunctors $p=(-)_{\chi},q=\id$ of $\QCoh(G)$ and the natural transformation $\eta=i_{\chi}$ introduced in Remark \ref{remark-homogeneous} form a diagram as above. In this situation, $\QCoh_{\chi}(G)$ is our $C^0$. Therefore, a morphism $G\to G'$ induces a morphism $\QCoh_\chi(G)\to \QCoh_\chi(G')$. Along the same lines, one proves that the association $$G\mapsto \QCoh_\id(G)$$ yields a well-defined functor $$\Ger_{\Gm}(X)\to \Lin(X).$$
	\end{rem}

We end this subsection by describing the symmetric monoidal structure on the category of $\Gm$-gerbes. 

    Let $X$ be a scheme and $G_1$ and $G_2$ be two $\Gm$-gerbes on $X$. One can construct the product $G_1 \star G_2$, which is a $\Gm$-gerbe such that its class in cohomology is the product of the classes of $G_1$ and $G_2$ (see \cite[Conjecture 5.23]{BP}). Clearly, this is not enough to define a symmetric monoidal structure on the category of $\Gm$-gerbes. The idea is to prove that this $\star$ product has a universal property in the $\infty$-categorical setting, which allows us to define the symmetric monoidal structure on the category of $\Gm$-gerbes using the theory of simplicial colored operads and $\infty$-operads (see Chapter 2 of \cite{HA}).

	\begin{construction}\label{star-product}
	
	Let $\AbGer(X)$ be the $(2,1)$-category of abelian gerbes over $X$ and $\AbGr(X)$ the $(1,1)$-category of sheaves of abelian groups over $X$. We have the so-called banding functor
	$$ \Band: \AbGer(X) \longrightarrow \AbGr(X),$$ 
 $$(\alpha:G\to X)\mapsto \alpha_*\cI_G$$
	(see \cite[(3.2)]{Bergh-Schnurer}). It is easy to prove that $\Band$ is symmetric monoidal with respect to the two Cartesian symmetric monoidal structures of the source and target, that is it extends to a symmetric monoidal functor
    $\Band: \AbGer(X)^{\times} \rightarrow \AbGr(X)^{\times}$ of colored simplicial operads (and therefore also of $\infty$-operads).
    
    Recall that, given a morphism of sheaf of groups $\phi:\mu \rightarrow \mu'$ and a $\mu$-gerbe $G$, we can construct a $\mu'$-gerbe, denoted by $\phi_*G$, and a morphism $\rho_{\phi}:G \rightarrow \phi_*G$ whose image through the banding functor is exactly $\phi$. This pushforward construction is essentially unique and verifies weak funtoriality. This follows from the following result: if $G$ is a gerbe banded by $\mu$, then the induced banding functor 
    $$\textup{Band}_{G/}: \AbGer(X)_{G/} \longrightarrow \AbGr(X)_{\mu/}$$
    is an equivalence and the pushforward construction is an inverse (see \cite[Proposition 3.9]{Bergh-Schnurer}). This also implies that $\Band$ is a coCartesian fibration.

    Let $\Fin_*$ be the category of pointed finite sets. Consider now the morphism $\Fin_* \rightarrow \AbGr(X)^{\times}$ induced by the algebra object $\mathbb{G}_{\textup{m},X}$ in $\AbGr(X)^{\times}$.\footnote{Note that, $\Gm$ being abelian, it is an algebra object in $\AbGr(X)^\times$, because the multiplication map is a morphism of groups.} We consider the following pullback diagram
    $$\begin{tikzcd}
    \mathcal{G} \arrow[d, "B"] \arrow[r] & \AbGer(X)^{\times} \arrow[d, "\Band"] \\
    \Fin_* \arrow[r, "\mathbb G_{\textup{m},X}"]       & \AbGr(X)^{\times} ;                  
    \end{tikzcd}$$
    where $B$ is again a coCartesian fibration. This implies that $\mathcal{G}$ is a symmetric monoidal structure over the fiber category $\mathcal{G}_{\langle 1\rangle}:=B^{-1}(\langle 1\rangle)$ which is exactly the category of $\Gm$-gerbes. The operadic nerve of $\mathcal{G}$ will be the symmetric monoidal $\infty$-category of $\Gm$-gerbes (see \cite[Proposition 2.1.1.27]{HA}).
    
    This symmetric monoidal structure coincides with the $\star$ product of gerbes defined in \cite[Construction 3.8]{Bergh-Schnurer}. Indeed, following the rigidification procedure in \cite[Exercise 12.F]{Olsson} and \cite[Appendix A]{AOV}, one can characterize the $\star$-product as the pushforward of the multiplication map $m:\Gm\times \Gm\to \Gm$, i.e. if $G_1$ and $G_2$ are two $\Gm$-gerbes over $X$, then $$G_1 \star G_2:= m_*(G_1 \times_X G_2).$$

    Because of this, one can also prove that under the identification $\Ger_{\Gm}(X)\simeq \Map(X,\B^2\Gm)$ from \cref{gerbes-H2} the $\star$-product coincides also with the product induced by the multiplication $\B^2\Gm\times \B^2\Gm\to \B^2\Gm$.
    \end{construction}
    \begin{rem}
	     For future reference, let us describe $\cG$ as a simplicial colored operad. The objects (or colors) of $\cG$ are $\Gm$-gerbes over $X$. Let $\{ G_i\}_{i \in I}$ be a sequence of objects indexed by a finite set $I$ and $\cH$ another object; we denote by $\prod_{i\in I}G_i$ the fiber product of $G_i$ over $X$. The simplicial set of multilinear maps $\textup{Mul}(\{ G_i\}_{i \in I}, \cH)$ is the full subcategory of the $1$-groupoid of morphisms  $\textup{Map}_X(\prod_{i \in I} G_i, \cH)$ of gerbes over $X$ such that its image through the banding functor is the $n$- fold multiplication map of $\Gm$, where $n$ is the cardinality of $I$. Note that because the simplicial sets of multilinear maps are Kan complexes by definition, then they are fibrant simplicial sets. This is why the operadic nerve gives us a symmetric monoidal structure, see \cite[Proposition 2.1.1.27]{HA}.
	\end{rem}
    \begin{defin}
		We denote by $\Ger_{\Gm}(X)^{\mathlarger\star}$ the symmetric monoidal $\infty$-groupoid of $\Gm$-gerbes over $X$, with the symmetric monoidal structure given by \cref{star-product}.
	\end{defin}

	\subsection{Reminders on stable and prestable linear categories}\label{section-prestable}
 In this subsection, we recall some definitions and needed properties in the context of linear prestable $\infty$-categories. Almost all notations, and all the results, are taken from \cite{SAG}: the purpose of this summary is only to gather all the needed vocabulary in a few pages, for the reader's convenience.
	\begin{defin}
		Let $\cc$ be an $\infty$-category. We will say that $\cc$ is \textit{prestable} if the following conditions are satisfied:
		\begin{itemize}\item The $\infty$-category $\cc$ is pointed and admits finite colimits.
			\item The suspension functor $\Sigma : \cc \to \cc$ is fully faithful.
			\item For every morphism $f : Y \to\Sigma Z$ in $\cc$, there exists a pullback and pushout square
			$$\begin{tikzcd} X\arrow[r, "f'"]\arrow[d] &Y\arrow[d, "f"]\\ 0 \arrow[r]&\Sigma Z.
			\end{tikzcd}
			$$
		\end{itemize}
	\end{defin}

	\begin{prop}[{\cite[Proposition C.1.2.9]{SAG}}]\label{prestable-as-connective}
		Let $\cc$ be an $\infty$-category. Then the following conditions are equivalent:
		\begin{itemize}\item $\cc$ is prestable and has finite limits.
			\item There exists a stable $\infty$-category $\cD$ equipped with a t-structure $(\cD_{\geq 0}, \cD_{\leq 0})$ and an equivalence $\cc\simeq \cD_{\geq 0}$.
		\end{itemize}
	\end{prop}
	
	\begin{prop}[{\cite[Proposition C.1.4.1]{SAG}}]\label{proposition-Grothendieck-cats}
		Let $\cc$ be a presentable $\infty$-category. Then the following conditions are equivalent:
		\begin{itemize}\item $\cc$ is prestable and filtered colimits in $\cc$ are left exact.
			\item There exists a presentable stable $\infty$-category $\cD$ equipped with a t-structure $(\cD_{\geq 0}, \cD_{\leq 0})$ compatible with filtered colimits, and an equivalence $\cc\simeq \cD_{\geq 0}$.
		\end{itemize}
	\end{prop}
	\begin{defin}[{\cite[Definition C.1.4.2]{SAG}}]\label{defin-Grothendieck-cats}Let $\cc$ be a presentable $\infty$-category. We will say that $\cc$ is \textit{Grothendieck}
		if it satisfies the equivalent conditions of \cref{proposition-Grothendieck-cats}. Following \cite[Definition C.3.0.5]{SAG}, we denote the $\infty$-category of Grothendieck presentable $\infty$-categories (and colimit-preserving functors between them) by $\Groth_\infty$. We also denote the category of presentable stable $\infty$-categories (and colimit-preserving functors between them) by $\Pr^{\textup{L}}_{\textup{St}}$.
	\end{defin}
	\begin{rem}
	By \cite{HA} and \cite[Theorem C.4.2.1]{SAG}, both $\Pr^\textup{L}_\textup{St}$ and $\Groth_\infty$ inherit a symmetric monoidal structure from $\Pr^\textup{L}$ which we denote again by $\otimes$. The stabilization functor $\textup{st}:\Pr^\textup{L}\to \Pr^\textup{L}_{\textup{St}}$ (and therefore its restriction to $\Groth$) is symmetric monoidal (this follows from \cite[Example 4.8.1.23]{HA} by keeping in mind that $\Sp\otimes \Sp\simeq \textup{st}(\Sp)\simeq \Sp$).
\end{rem}
	\begin{defin}Let $X$ be a qcqs scheme. An $\cO_X$-linear prestable $\infty$-category is an object of $$\Mod_{\QCoh(X)_{\geq 0}}(\Groth_\infty^\otimes).$$
	
	A stable presentable $\cO_X$-linear $\infty$-category is an object of $$\Mod_{\QCoh(X)}(\Pr^{\textup L, \otimes}_{\textup{St}}).$$
		\end{defin}

		\begin{rem}\label{rem-product}The category $\Mod_{\QCoh(X)_{\geq 0}}(\Groth_\infty^\otimes)$ has a tensor product $-\otimes_{\QCoh(X)_{\geq 0}}-$ (which we will abbreviate by $\otimes$) induced by the Lurie tensor product of presentable $\infty$-categories. See \cite[Theorem C.4.2.1]{SAG} and \cite[Section 10.1.6]{SAG} for more details. The same is true for $\Mod_{\QCoh(X)}(\Pr^{\textup L, \otimes}_{\textup{St}})$. There is a naturally induced stabilization functor $$\Mod_{\QCoh(X)_{\geq 0}}(\Groth_\infty^\otimes)\to \Mod_{\QCoh(X)}(\Pr^{\textup L, \otimes}_{\textup{St}})$$ which is symmetric monoidal with respect to these structures (again by an argument analogous to the one sketched in the nonlinear setting). \end{rem}

	$\Mod_{\QCoh(X)_{\geq 0}}(\Groth_\infty^\otimes)$ and $\Mod_{\QCoh(X)}(\Pr^{\textup L, \otimes}_{\textup{St}})$ satisfy a very important ``descent'' property, which is what Gaitsgory \cite{Gaitsgory-affineness} calls 1-affineness.
	
	\begin{construction}
		The functors $$\textup{CAlg}_k\to \Cat_\infty$$
		$$R\mapsto \Mod_{\QCoh(X)_{\geq 0}}(\Groth_\infty^\otimes)$$
		$$R\mapsto \Mod_{\QCoh(X)}(\Pr^{\textup L, \otimes}_{\textup{St}})$$
		can be right Kan extended to functors $$\QStk,\QStkpre:\Sch_k^\textup{op}\to \Cat_\infty.$$ This gives a meaning to the expressions $\QStk(X),\QStkpre(X)$, which can be thought of as ``the category of sheaves of $\QCoh$-linear (resp. $\QCoh_{\geq 0}$-linear) (pre)stable categories on $X$''. For any $X\in \Sch_k$, there are well-defined ``global sections functors'' $$\QStk(X)\to\Mod_{\QCoh(X)}(\Pr^{\textup L, \otimes}_{\textup{St}}),$$$$ \QStkpre(X)\to \Mod_{\QCoh(X)_{\geq 0}}(\Groth_\infty^\otimes)$$ constructed in \cite[discussion before Theorem 10.2.0.1]{SAG}.
		
		\end{construction}
	\begin{thm}\label{stack-of-categories}Let $X$ be a qcqs scheme over $k$. Then the global sections functors $$\QStk(X)\to \Mod_{\QCoh(X)}(\Pr^{\textup L, \otimes}_{\textup{St}})$$ $$\QStkpre(X)\to \Mod_{\QCoh(X)_{\geq 0}}(\Groth_\infty^\otimes)$$ are symmetric monoidal equivalences.\end{thm}
	\begin{proof} For the stable part, this is \cite[Proposition 6.5]{DAG-XI}. For the prestable part, this is the combination of \cite[Theorem 10.2.0.2]{SAG} and \cite[Theorem D.5.3.1]{SAG}. Symmetric monoidality follows from the fact that the inverse of the global sections functor (the ``localization functor'', see \cite[Section 2.3]{BP}) is strong monoidal.
	\end{proof}
	This theorem means that, if $X$ is a qcqs scheme over $k$, every (Grothendieck pre)stable presentable $\cO_X$-linear $\infty$-category $\cc$ has an associated sheaf of $\infty$-categories on $X$ having $\cc$ as category of global sections. We will make substantial use of this fact in the present work. The main reason why we will always assume our base scheme $X$ to be qcqs is because it makes this theorem hold.
	\begin{defin}\label{defin-prestable-stable-linear}We denote the two right-hand sides of the equivalences in \cref{stack-of-categories} respectively by $$\Lin(X)$$ and $$\Linpre(X).$$\end{defin}
		
	\begin{defin}\label{defin-derived-Brauer}We denote by $\cBr(X)$ the maximal $\infty$-groupoid contained in $\Linpre(X)$ and generated by $\otimes_{\QCoh(X)_{\geq 0}}$-invertible objects which are compactly generated categories, and equivalences between them. We denote by $\cBre(X)$ the maximal $\infty$-groupoid contained in $\Lin(X)$ generated by $\otimes_{\QCoh(X)}$-invertible objects which are compactly generated categories, and equivalences between them.
	
	We call $\Br(X):=\pi_0\cBr(X)$ the \textit{derived Brauer group} of $X$ and $\Bre(X):=\pi_0(\cBre(X))$ the \textit{extended derived Brauer group} of $X$.\footnote{Our notation here slightly differs from the one used in \cite[Definition 11.5.2.1, Definition 11.5.7.1]{SAG}, in that we write $\Br(X),\Bre(X)$ in place of $\textup{Br}(X),\textup{Br}^\dagger(X)$ in order to avoid confusion with the classical Brauer group. For the same reason, we introduce the terminology ``derived Brauer group'' (resp. ``extended derived Brauer group'') which interpolates between the one used by Lurie and the one appearing in \cite[Definition 2.14]{Toen-Azumaya}. Indeed, our extended derived Brauer group $\Bre(X)$ is the same as what To\"en in \textit{loc. cit.} calls $\textup{dBr}_{\textup{cat}}(X)$, the \textit{derived categorical Brauer group} of $X$, whereas there is no notation corresponding to $\Br(X)$ in To\"en's paper.}
	\end{defin}
	
	By \cite[Theorem 10.3.2.1]{SAG}, to be compactly generated is a property which satisfies descent. The same is true for invertibility, since the global sections functor is a symmetric monoidal equivalence by \cref{stack-of-categories}.

 \begin{thm}[{\cite[Example 11.5.7.15 and Example 11.5.5.5]{SAG}}]\label{cohomological-interpretation-brauer}
     Let $X$ be a qcqs scheme. There are equivalences of spaces $$\cBr(X)\simeq \Map(X, \tK(\Gm, 2))$$ and $$\cBre(X)\simeq \Map(X,\tK(\Gm,2)\times \tK(\Z, 1))$$ and therefore bijections $$\Br(X)\leftrightarrow \tH^2(X,\Gm)$$
     $$\Bre(X)\leftrightarrow \tH^2(X,\Gm)\times \tH^1(X,\Z),$$ which can be promoted to isomorphisms of abelian groups.
 \end{thm}

	\begin{rem}\label{t-structures}
	     By See \cite[Remark 11.5.7.3]{SAG}, the stabilization functor mentioned in \cref{rem-product} restricts to a functor $\cBr(X)\to \cBre(X)$, whose homotopy fiber at any object $\cc\in \cBre(X)$ is discrete and can be identified with the collection of all t-structures $(\cc_{\geq 0},\cc_{\leq 0})$ on $\cc$ satisfying the following
conditions:
\begin{itemize}\item The t-structure $(\cc_{\geq 0},\cc_{\leq 0})$ is right complete and compatible with filtered colimits.
\item The Grothendieck prestable $\infty$-category $\cc_{\geq 0}$ is an $\otimes$-invertible object of $\Linpre(X)$.
\item The Grothendieck prestable $\infty$-category $\cc_{\geq 0}$, and consequently its $\otimes$-inverse, are compactly generated\footnote{If a presentable category $\mathcal D$ is dualizable and compactly generated, then its dual can be presented as $\textup{Ind}((\mathcal D^\omega)^\textup{op})$, where $\mathcal D^\omega$ is the subcategory of compact objects. In our case, the inverse is in particular the dual.}.\end{itemize}

However, the stabilization functor is faithful (\cite[Proposition C.3.1.1]{SAG}), and the induced map of groups $$\Br(X)\to \Bre(X)$$ is injective, in that it corresponds via \cref{cohomological-interpretation-brauer} to the injection $$\tH^2_\ett(X,\Gm)\to \tH^2_\ett(X,\Gm)\times\tH^1_\ett(X,\Z)$$ induced by the zero element of $\tH^1_\ett(X,\Z)$. We do not prove this last assertion, which however follows easily from the results in \cite[Section 11.5]{SAG}, since the relationship between the Brauer space and the extended Brauer space from the point of view of gerbes and torsors will be analyzed in a future work. We will never make use of this result in our proofs.
	\end{rem}

\begin{rem}
We now describe the symmetric monoidal structure of $\Lin(X)$ using the language of simplicial colored operads. The same description will apply to the prestable case. This will come out useful in the rest of the paper.
    
The simplicial colored operad $\Lin(X)$ can be described as follows:
\begin{enumerate}
    \item the objects (or colors) are stable presentable $\cO_X$-linear $\infty$-categories;
    \item given $\{ M_i \}_{i \in I}$ a sequence of objects indexed by a finite set $I$ and $N$ another object, the simplicial set of multilinear maps $\textup{Mul}(\{ M_i\}_{i \in I}, N)$ is the full subspace of $\textup{Fun}(\prod_{i \in I}M_i,N)$ spanned by functors between stable presentable $\infty$-categories which are $\QCoh(X)$-linear and preserve small colimits separately in each variable.
\end{enumerate} 
\end{rem}

 \begin{rem}\label{t-structure-tensor}
     Following \cite[Remark C.4.2.2]{SAG}, one can put a t-structure on the tensor product of two presentable categories $\cc$ and $\cD$ each one equipped with a t-structure (and similarly in the $\cO_X$-linear setting). Again by \cite[Remark C.4.2.2]{SAG}, we have the following characterization: $$(\cc\otimes\cD)_{\geq 0}$$ is the smallest full subcategory of $\cc\otimes \cD$ containg $\cc_{\geq 0}\otimes \cD_{\geq 0}$. The analogue statement in the $\cO_X$-linear setting implies that, for $G,G'\to X$ $\Gm$-gerbes over $X$, $$(\QCoh(G)\otimes_{\QCoh(X)}\QCoh(G'))_{\geq 0}\simeq \QCoh(G)_{\geq 0}\otimes_{\QCoh(X)}\QCoh(G')_{\geq 0}.$$
 \end{rem}

 \subsection{The correspondence}\label{outline}
    We now explain in detail the statement of \cref{main-theorem} and offer a sketch of its proof, which will be carried out in \cref{Section-2}.
    
  The construction of the $\infty$-category of twisted sheaves (see \cref{twisted-sheaves}) gives rise to a functor $$\Ger_{\Gm}(X)\to \cBre(X)\subset \Lin(X)$$
	$$G\mapsto \QCoh_{\id}(G).$$ (the fact that it takes values in $\cBre(X)$ is proven both in \cref{id-is-invertible} and \cite[Example 9.3]{Hall-Rydh}).
 
	Under the identifications $$\Ger_{\Gm}(X)\simeq\Map(X,\B^2\Gm)$$ (\cref{gerbes-H2}) and $$\cBre(X)\xrightarrow{\sim} \Map(X, \B^2\Gm\times \B\Z),$$
 (\cite[Example 11.5.5.5]{SAG}) this functor corresponds to a section of the map $$\Map(X, \B^2\Gm\times \B\Z)\to \Map(X,\B^2\Gm).$$ This section is neither fully faithful nor essentially surjective. However, one can observe that $\QCoh_\id(G)$ is not just an $\cO_X$-linear stable $\infty$-category, but also carries a t-structure which is compatible with filtered colimits, since as recalled in \cref{remark-homogeneous} the functor $(-)_\id:\QCoh(G)\to \QCoh(G)$ is t-exact. This additional datum allows to ``correct'' the fact that $\QCoh_\id(-)$ is not an equivalence. 
 
 More precisely, the fact that $\QCoh(G)$ has a functorial t-structure follows from \cite[Example 10.1.6.2, 10.1.7.2]{SAG}. One defines $$\QCoh(G)_{\geq 0}\simeq \lim_{\Spec A\to G}(\Mod_A)_{\geq 0}$$ (the transitions $(A\to B)\mapsto (\Mod_A\to \Mod_B$) are well-defined because $M\mapsto M\otimes B$ is left exact). Then, we define $$\QCoh_\id(G)_{\geq 0}=\QCoh(G)_{\geq 0}\cap \QCoh_\id(G)$$ where the intersection is understood via the fully faithful functor $\QCoh(G)_{\geq 0}\to \QCoh(G)$. Functoriality of this last operation follows from the functorialities made explicit in \cref{equalizer}.
 
 \begin{defin}\label{functor-psi}We denote by $$\Psi:\Ger_{\Gm}(X)\to \Linpre(X)$$ the functor assigning
	$$G\mapsto \QCoh_\id(G)_{\geq 0}.$$\end{defin}

 In analogy to the stable setting, we will prove in \cref{id-is-invertible} that $\Psi$ factors through the (not full) subcategory $\cBr(X)\to \Linpre(X)$.

 	\begin{rem}Given a morphism of gerbes $f:G\to G'$, there are induced functors $$\QCoh(G')\xrightarrow{f^*} \QCoh(G)$$ and $$\QCoh(G)\xrightarrow{f_*}\QCoh(G).$$ We will always use the covariant functoriality. We omit the proof of the fact that $f_*$ sends twisted sheaves to twisted sheaves, which follows from a straightforward computation.
	
	Although there are no real issues with the contravariant notation, there are some technical complications when one wants to prove that a contravariant functor is symmetric monoidal.\footnote{In fact, one would need to use the construction of the ``opposite of an $\infty$-operad'', see \cite{MO-opposite-operad}, \cite[Remark 2.4.2.7]{HA}.} Note that, if $f$ is any morphism of gerbes (that commutes with the banding), then $f$ is an isomorphism, therefore $f_*=(f^*)^{-1}$.
	\end{rem}
	
	 The inverse to $\Psi$ will be described as follows. By \cref{stack-of-categories}, if $X$ is a quasicompact quasiseparated scheme, and $M$ a prestable presentable $\infty$-category equipped with an action of $\QCoh(X)_{\geq 0}$ in $\Pr^{\textup{L}}$, then $M$ is the category of global sections over $X$ of a unique sheaf of prestable presentable $\QCoh(X)_{\geq 0}$-linear categories $\cM$.
	
	\begin{defin}\label{defin-triv}
		Let $X$ be a quasicompact quasiseparated scheme, and $M$ be an element of $\cBr(X)$. Then we define the $\Triv_{\geq 0}(M)$ as the stack $$S \mapsto\Equiv\Big( \QCoh(S)_{\geq 0},\cM(S)\Big)$$ where the right-hand side is the space of equivalences of $\infty$ categories between $\QCoh(S)_{\geq 0}$ and $\cM(S)$. This is a stack because both $\cM$ and $\QCoh(-)_{\geq 0}$ are, and by \cref{stack-of-categories}.
  The functor $$\cBr(X)\to \Stk_{/X}$$
  $$M\mapsto \Triv_{\geq 0}(M)$$ will be denoted by $\Phi$.
	\end{defin}

 We will prove in \cref{triv-gerbe} that $\Phi$ factors through the ``forgetful'' functor $\Ger_{\Gm}(X)\to \Stk_{/X}$ (we will abuse notation and denote the resulting functor again by $\Phi$).
 
 Note that both $\cBr(X)$ and $\Ger_{\Gm}(X)$ have symmetric monoidal structures: on the first one, we have the tensor product of linear $\infty$-categories $-\otimes_{\QCoh(X)_{\geq 0}}-$ (\cref{rem-product}), and on the second one we have the rigidified product $\star$ of $\Gm$-gerbes (\cref{star-product}).
 
As anticipated, our main result is the following:

 \begin{thm}\label{main-theorem}
		Let $X$ be a qcqs scheme over a field $k$. Then there is a symmetric monoidal equivalence of $\infty$-groupoids\footnote{Actually, both sides are 2-truncated, see \cref{2-groupoids}.}
		$$\Phi:\cBr(X)^{\otimes}\xleftrightarrow{\sim} \Ger_{\Gm}(X)^{\mathlarger \star}:\Psi$$
		where
		\begin{itemize}
			\item $\Phi(M)=\Triv_{\geq 0}(M)$
			\item $\Psi(G)=\QCoh_\id(G)_{\geq 0}$.
		\end{itemize}
	\end{thm}
 As a corollary of the proof, we also have that 
 \begin{cor}[{\cref{Conjecture-holds}}]
     Conjecture 5.27 in \cite{BP} is true.
 \end{cor}

 We emphasize that symmetric monoidality is somehow a special feature of the prestable setting. Indeed, the analogous equivalence $\cBre(X)\simeq \Map(X,\B^2\Gm\times \B\Z)$ is not symmetric monoidal, although it becomes an isomorphism of abelian groups after passing to $\pi_0$ (\cite[Remark 11.5.5.4,11.5.5.5]{SAG}). Moreover, we use the symmetric monoidal structures on the two functors in a substantial way in order to prove other properties.

\

 \cref{Section-2} is devoted to the proof of \cref{main-theorem}. Here is a sketch of our arguments.
 
 First, we establish the sought symmetric monoidal structure on $\Psi$ (\cref{id-is-invertible}). We do this by considering the $\infty$-category of bi-homogeneous sheaves over $G\times_X G'$ in the sense of \cref{bi-homogeneous} and proving that the external tensor product establishes an equivalence between $\QCoh_\chi(G)\otimes_{\QCoh(X)}\QCoh_{\chi'}(G')\simeq \QCoh_{(\chi,\chi')}(G\times_XG')$ (\cref{external-product-theorem}). This is done by reducing to the case of trivial gerbes (i.e. gerbes of the form $\BGm\times X$) and by using the fact that $\QCoh(\BGm)$ is compactly generated. Then we prove that, in the special case of $\chi=\chi'=\id$, the universal (``rigidification'') map $\rho:G\times_X G'\to G\star G'$ induces an equivalence $\QCoh_{(\id,\id)}(G\times_XG')\simeq \QCoh_\id(G\star G')$ (\cref{rigidification-pullback-theorem}): the fact that $\rho^*$ sends twisted sheaves to $(\id,\id)$-homogeneous sheaves is a direct computation using the universal property of the $\star$-product and the behaviour of the line bundles $\cL_\id, \cL_{(\id,\id)}$, whereas the fact that it is an equivalence is checked \'etale-locally.
 
The above equivalences are both t-exact, because $\rho:G \times_X G'\to G\star G'$ is representable and flat and thus one can apply t-exactness of $\rho^*$ on quasicoherent sheaves. Therefore, all results can be rephrased in the connective setting, i.e. for $\QCoh_\id(G)_{\geq 0}$.
In particular, since any $\Gm$-gerbe $G$ is $\star$-invertible, the existence of a symmetric monoidal structure on $\Psi$ implies that $\QCoh_\id(G)_{\geq 0}$ belongs to $\cBr(X)$.

Together, \cref{external-product-theorem} and \cref{rigidification-pullback-theorem} prove \cite[Conjecture 5.27]{BP}.

We then pass to proving that, for any $M\in \cBr(X)$, $\Phi(M)$ is indeed a $\Gm$-gerbe (\cref{triv-gerbe}), which follows essentially by working \'etale-locally on $X$ and inspecting $\QCoh(X)_{\geq 0}$-linear automorphisms of $\QCoh(X)_{\geq 0}$ itself. We also establish a symmetric monoidal structure on $\Phi$ (\cref{symm-mon}), which in this case follows directly from the universal properties defining $\otimes$ and $\star$.
In \cref{section-equivalence}, we conclude the proof of the main theorem. We start by proving that $\Phi$ is fully faithful (\cref{Triv-is-ff}), which amounts again to a computation of automorphisms and uses the symmetric monoidal structure on $\Phi$ to easily reduce to working with the unit of $\cBr(X)$, i.e. $\QCoh(X)_{\geq 0}$. Finally, we prove that there is a natural equivalence $$\textup{Id}_{\Ger_{\Gm}(X)}\Rightarrow \Phi\Psi.$$ The existence of a natural transformation is a consequence of the definitions, while the fact that it is an equivalence follows again by \'etale-local arguments.
	
\section{Study of the derived Brauer map}\label{Section-2}

 \subsection{Derived categories of twisted sheaves}\label{QCoh-id-cn-is-invertible}

	Our aim in this subsection is to prove the following statement:
	\begin{thm}\label{id-is-invertible}Let $X$ be a qcqs scheme. The functors $$\QCoh_\id(-):\Ger_{\Gm}(X)^{\mathlarger\star}\to \Lin(X)^\otimes$$ and $$\QCoh_\id(-)_{\geq 0}:\Ger_{\Gm}(X)^{\mathlarger\star}\to \Linpre(X)^{\otimes}$$ carry a symmetric monoidal structure with respect to the $\star$-symmetric monoidal structure on the left hand side and to the $\otimes$-symmetric monoidal structures on the right hand sides. In particular, since every $\Gm$-gerbe is $\star$-invertible and $\Ger_{\Gm}(X)$ is a $2$-groupoid, $\QCoh_\id(-)$ takes values in $\cBre(X)$ and $\QCoh_\id(-)_{\geq 0}$ takes values in $\cBr(X)$.
	\end{thm}

	We will prove \cref{id-is-invertible} in two different steps, \cref{external-product-theorem} and \cref{rigidification-pullback-theorem}, which also prove \cite[Conjecture 5.27]{BP}.

	\begin{prop}\label{external-product-theorem}Let $X$ be a quasicompact quasiseparated scheme over a field $k$.	Let $G,G'\to X$ be two $\Gm$-gerbes over $X$, and $\chi, \chi'$ two characters of $\Gm$. The external tensor product establishes a t-exact equivalence $$\boxtimes: \QCoh_\chi(G)\otimes_{\QCoh(X)}\QCoh_{\chi'}(G')\xrightarrow{\sim}\QCoh_{(\chi,\chi')}(G\times_XG')$$ where $G\times_XG'$ is seen as a $\Gm\times \Gm$-gerbe on $X$.
	\end{prop}
	\begin{proof}

		First we prove that the external tensor product induces a t-exact equivalence \begin{equation}\label{external-tens}\boxtimes:\QCoh(G)\otimes_{\QCoh(X)}\QCoh(G')\simeq \QCoh(G\times_XG').\end{equation} By \cref{stack-of-categories}, every side of the sought equivalence is the $\infty$-category of global sections of a sheaf in categories over $X$, which means that the equivalence is \'etale-local on $X$. Therefore, by choosing a covering $U\to X$ which trivializes both $G$ and $G'$, we can reduce to the case $G=U\times \BGm\to U, G'=U \times \BGm\to U$, both maps being the projection to $U$. But by \cite[Corollary 5.4]{BP} (or more generally \cite[Corollary 9.4.2.4]{SAG}), $\QCoh(U\times \BGm)\simeq \QCoh(U)\otimes \QCoh(\BGm)$, and this proves our claim. Finally, t-exactness follows from \cref{t-structure-tensor} and the K\"unneth formula.

		Now we prove that the restriction of the functor \eqref{external-tens} to $\QCoh_\chi(G)\otimes_{\QCoh(X)}\QCoh_{\chi'}(G')$ lands in $\QCoh_{(\chi, \chi')}(G\times_X G').$ It is enough to construct an equivalence
		$$ \alpha(\cF\boxtimes \cF'): (\cF\boxtimes \cF')_{(\chi,\chi')} \simeq (\cF)_\chi\boxtimes (\cF')_{\chi'}$$
		for every $\cF \boxtimes \cF'$. With the notation of Remark \ref{remark-homogeneous}, we can do this using the following chain of equivalence:
		\begin{equation*}
		\begin{split}
		(\cF\boxtimes \cF')_{(\chi,\chi')} &=(p\times p')_*(\act^*_{(\alpha,\alpha')}(\cF\boxtimes\cF')\otimes \cL_{(\chi,\chi')}^\vee) \\ & \simeq (p\times p')_*((\act_{\alpha}\times\act_{\alpha'})^*(\cF\boxtimes\cF')\otimes (\cL_{\chi}^\vee\boxtimes\cL_{\chi'}^\vee))\\ &\simeq (p\times p')_*((\act^*_{\alpha}\cF\otimes \cL_{\chi}^\vee)\boxtimes(\act_{\alpha'}^*\cF'\otimes \cL_{\chi'}^\vee))\simeq (\cF)_\chi\boxtimes (\cF')_{\chi'}.
		\end{split}
	    \end{equation*}
		A straightforward computation shows that $\alpha$ is in fact a natural transformation and verifies the condition described in Remark \cref{equalizer}. 
	
		It remains to prove that the restricted functor induces an equivalence, which will automatically be t-exact. But again, it suffices to show this locally, and in the local case this just reduces to the fact that $$\QCoh_{\chi}(U\times\BGm)\otimes_{\QCoh(U)}\QCoh_{\chi'}(U\times \BGm)\simeq \QCoh(U)\otimes_{\QCoh(U)}\QCoh(U)$$$$\simeq \QCoh(U)$$ because $\QCoh_{\chi}(U\times\BGm)\simeq \QCoh(U)$ thanks to \cite[Lemma 5.20, (3)]{BP}.\end{proof}
		
	\begin{rem}
	Using the associativity of the box product $\boxtimes$, the same proof works in the case of a finite number of $\Gm$-gerbes.
	\end{rem}

	\begin{prop}\label{rigidification-pullback-theorem}Let $G\star G'$ be the $\Gm$-gerbe defined in \cref{star-product}. Then the pullback along $\rho:G\times_XG'\to G\star G'$ establishes a t-exact equivalence $$\QCoh_{\id}(G\star G')\xrightarrow{\sim}\QCoh_{(\id,\id)}(G\times_X G').$$
	\end{prop}
 \begin{proof}The proof is a direct application of the universal property of the $\star$-product, together with the fact that the structure map $\rho$ is flat (\cite[Theorem A.1]{AOV}). 
 
 Note first that the character $m:\Gm\times \Gm\to \Gm$ given by the multiplication induces a line bundle $L_{\id,\id}$ on $\BGm\times \BGm$. One can prove easily that this line bundle coincides with the external product of two copies of $L_\id$, the universal line bundle on $\BGm$. This means that what we denote by $\cL_{(\id,\id)}\in \QCoh(G\times_XG'\times \BGm\times \BGm)$ has the form $\cL_\id\boxtimes \cL_\id$ (again with the usual notations of Remark \ref{remark-homogeneous}).
		
		We need to prove that $\rho^*$ sends $\id$-twisted sheaves in $(\id,\id)$-twisted sheaves. To do this, we will construct an equivalence 
		$$ \alpha(\cF):(\rho^*\cF)_{(\id,\id)} \simeq \rho^*(\cF_{\id})$$ 
		for every $\cF$ in $\QCoh(G \star G')$ and apply \cref{equalizer}. An easy computation shows that $\alpha$ is in fact a natural transformation and it verifies the condition described in \cref{equalizer}.
		
		By construction of the $\star$ product, one can prove that the following diagram 
		$$
		\begin{tikzcd}
			G \times_X G'\times \BGm \times \BGm \arrow[rr,                  "{\act_{(\alpha,\alpha')}}"] \arrow[d, "{(\rho, \textup{B}m)}"] &  &               G\times_X G'  \arrow[d, "\rho"] \\
			G \star G' \times \BGm \arrow[rr, "\act_{\alpha\alpha'}"]                                                   &  & G \star G'                     
		\end{tikzcd}
		$$
		is commutative, where $\act_{(\alpha,\alpha')}$ is the action map of $G\times_X G'$ defined by the product banding, $\act_{\alpha\alpha'}$ is the action map of $G\star G'$ and $\textup{B}m$ is the multiplication map $m:\Gm \times \Gm \rightarrow \Gm$ at the level of classfying stacks. This implies that 
		$$\act^*_{(\alpha,\alpha')}(\rho^*\cF) =(\rho,\textup{B}m)^*\act^*_{\alpha\alpha'}(\cF).$$ 
		
		Consider now the following diagram:
		$$
		\begin{tikzcd}
			G \times_X G'\times \BGm \times \BGm \arrow[rdd, "{(\rho,\textup{B}m)}"', bend right] \arrow[rrrd, "{\tilde{q}_{(\alpha,\alpha')}}", bend left] \arrow[rd, "{(\id,\textup{B}m)}"] &                                                                                         &  &                                 \\
			& G \times_X G'\times \BGm \arrow[rr, "{q_{(\alpha,\alpha')}}"] \arrow[d, "{(\rho,\id)}"] &  & G\times_X G'  \arrow[d, "\rho"] \\
			& G \star G' \times \BGm \arrow[rr, "q_{\alpha\alpha'}"]                                 &  & G\star G'                      
		\end{tikzcd}
		$$
		where $\tilde{q}_{(\alpha,\alpha')},q_{(\alpha,\alpha')}$ and $q_{\alpha\alpha'}$ are the projections. This is a commutative diagram and the square is a pullback. 
		
		Finally, we can compute the $(\id,\id)$-twisted part of $\rho^*\cF$: 
		\begin{equation*}
			\begin{split}
				(\rho^*\cF)_{\id,\id}=& \tilde{q}_{(\alpha,\alpha'),*}(\act^*_{(\alpha,\alpha')}(\rho^*\cF)\otimes \cL_{\id,\id}^{\vee}) \\ =&
				q_{(\alpha,\alpha'),*}(\id,\textup{B}m)_*\Big( (\id,\textup{B}m)^*(\rho,\id)^*(\act_{\alpha\alpha'}^*(\cF))\otimes \cL_{\id,\id}^{\vee} \Big);
			\end{split}
		\end{equation*}
		notice that since $\textup{B}m^* L_{\id} = L_{\id,\id}$ we have $(\rho,\textup{B}m)^*\cL_{\id}=\cL_{\id,\id}$. Furthermore, an easy computation shows that the unit $\id \rightarrow (\id,\textup{B}m)_*(\id,\textup{B}m)^*$ is in fact an isomorphism,  because of the explicit description of the decomposition of the stable $\infty$-category of quasi-coherent sheaves over $\BGm$. These two facts together give us that 
		\begin{equation*}
			\begin{split}
				(\rho^*\cF)_{\id,\id}=& 
				q_{(\alpha,\alpha'),*}(\id,\textup{B}m)_*\Big( (\id,\textup{B}m)^*(\rho,\id)^*(\act_{\alpha\alpha'}^*(\cF))\otimes \cL_{\id,\id}^{\vee} \Big) \\ =&  q_{(\alpha,\alpha'),*}(\id,\textup{B}m)_*(\id,\textup{B}m)^*(\rho,\id)^*\Big( \act_{\alpha\alpha'}^*(\cF)\otimes \cL_{\id}^{\vee} \Big) \\ =&
				q_{(\alpha,\alpha'),*}(\rho,\id)^*\Big(\act^*_{\alpha\alpha'}(\cF)\otimes \cL_{\id}^{\vee}\Big) \\ =& 
				\rho^*q_{\alpha\alpha',*}\Big(\act^*_{\alpha\alpha'}(\cF)\otimes \cL_{\id}^{\vee}\Big) \\ =& 
				\rho^*( (\cF)_{\id} ).
			\end{split}
		\end{equation*}

		To finish the proof, we need to verify that $\rho^*$ restricted to twisted sheaves is an equivalence with the category of $(\id,\id)$-homogeneous sheaves.  Again, this can be checked \'etale locally, therefore we can reduce to the case $G\simeq G' \simeq X \times \BGm$ where the morphism $\rho$ can be identified with $(\id_X,\textup{B}m)$. We know that for the trivial gerbe we have that $\QCoh(X) \simeq \QCoh_{\id}(X\times \BGm)$ where the map is described by $ \cF \mapsto \pi^*\cF \otimes \cL_{\id}$, $\pi$ being is the structural morphism of the (trivial) gerbe. A straightforward computation shows that, using the identification above, the morphism $\rho^*$ is the identity of $\QCoh(X)$.
		
		Finally, t-exactness follows from the fact that $\rho$ is flat and representable (by \cite[Theorem A.1]{AOV}) and by \cite[Remark 9.1.3.4]{SAG}, a stacky and derived generalization of \cite[Tag 02KH]{Stacks}.
	\end{proof}
	
	\begin{rem}\label{remark-theorem}
	One can prove easily that $\rho_*$ is an inverse of the morphism $\rho^*$ which we have just described. It is still true that it sends twisted sheaves to twisted sheaves and that it has the same functorial property of the pullback, due to the natural adjunction. Furthermore, $\rho_*$ is t-exact because $\rho$ is a morphism of gerbes whose image through the banding functor $\textup{Band}(\rho)$ is a surjective morphism of groups with a linearly reductive group as kernel. This implies that it is the structure morphism of a gerbe banded by a linearly reductive group, therefore $\rho_*$ is exact.
	\end{rem}

 \begin{rem}\label{Conjecture-holds}As a direct consequence of \cref{external-product-theorem} and \cref{rigidification-pullback-theorem}, \cite[Conjecture 5.27]{BP} holds.\end{rem}

	\begin{proof}[Proof of \cref{id-is-invertible}]
	Na\"ively speaking, the proof amounts to combining \cref{external-product-theorem} and \cref{rigidification-pullback-theorem}. First of all, we start with the stable case. We have to lift $\QCoh_{\id}(-)$ from a morphism of $\infty$-groupoids to a symmetric monoidal functor, i.e. to extend the action of $\QCoh_{\id}(-)$ to multilinear maps. Let $\{ G_i \}_{i \in I}$ be a sequence of $\Gm$-gerbes indexed by a finite set $I$ and $H$ be a $\Gm$-gerbe, then we can define a morphism of simplicial set
	$$\QCoh_{\id}(-): \textup{Mul}(\{ G_i \}_{i \in I}, H) \longrightarrow \textup{Mul}(\{ \QCoh_{\id}(G_i) \}_{i \in I}, \QCoh_{\id}(H))$$ 
	by the following rule: if $f:\prod_{i \in I} G_i \rightarrow H$ is a multilinear map, we define $\QCoh_{\id}(f)$ to be the composition $f_* \circ \boxtimes^I: \prod_{i \in I} \QCoh_{\id}(G_i) \rightarrow \QCoh_{\id}(H)$ where $\boxtimes^I:\prod_{i \in I} \QCoh_{\id}(G_i) \rightarrow \QCoh_{\id}(\prod_{i \in I}G_i)$ is the $I$-fold box product. The fact that this association is well-defined follows from Proposition \ref{external-product-theorem} and Proposition \ref{rigidification-pullback-theorem}. A priori, $\QCoh_{\id}(-)$ is a lax-monoidal functor from $\Ger_{\Gm}(X)^{\star}$ to $\Lin(X)^{\times}$, where $\Lin(X)^{\times}$ is the symmetric monoidal structure on $\Lin(X)$ induced by the product of $\infty$-categories. However, $f_*$ and $\boxtimes^I$ are both $\QCoh(X)$-linear and preserves small colimits. Notice that $f_*$ preserves small colimits because it is exact, due to the fact that it is a $\mu$-gerbe, with $\mu$ a linearly reductive group (see Remark \ref{remark-theorem}). This implies that $\QCoh_{\id}(-)$ can be upgraded to a morphism of $\infty$-operads if we take the operadic nerve. It remains to prove that it is symmetric monoidal, i.e. to prove that it sends coCartesian morphism to coCartesian morphism. This follows again from the isomorphisms described in Proposition \ref{rigidification-pullback-theorem} and Proposition \ref{external-product-theorem}.

	The prestable case can be dealt with in the exact same way. Because the equivalences in Proposition \ref{external-product-theorem} and Proposition \ref{rigidification-pullback-theorem} are t-exact, they restrict to the prestable connective part of the $\infty$-categories. The fact that they remain equivalences can be checked \'etale locally.
	\end{proof}

		\subsection{Gerbes of connective trivializations}\label{Triv-is-gerbe}
	
		Let $M\in \cBr(X)$. We recall that the functor $$\Triv_{\geq 0}(M):\Sch_{/X}\to \cS$$
		is defined as $$(S\to X)\mapsto \Equiv_{\QCoh(S)_{\geq 0}}(\QCoh(S)_{\geq 0},\cM(S))$$ where $\cM$ is the stack of categories associated to $M$ (see \cref{stack-of-categories}).
	
	Our aim in this subsection is to prove that for every $M\in \cBr(X)$, the functor $\Triv_{\geq 0}(M)$ has a natural structure of a gerbe over $X$, and also that the functor $\Triv_{\geq 0}(-):\cBr(X)\to \Ger_{\Gm}(X)$ can be promoted to a symmetric monoidal functor $\cBr(X)^\otimes\to \Ger_{\Gm}(X)^{\mathlarger\star}$.
	Let us recall the main result about the Brauer space proven in \cite{SAG} (specialized to the case of qcqs schemes).
	\begin{thm}[{\cite[Theorem 11.5.7.11]{SAG}}]\label{local-triviality-prestable}Let $X$ be a qcqs scheme. Then for every $M\in \Br(X)=\pi_0(\cBr(X))$, there exists an \'etale covering $f:U\to X$ such that $f^*M=0$ in $\Br(U)$; that is, any representative of $f^*M$ is equivalent to $\QCoh(U)_{\geq 0}$ as an $\infty$-category.
	\end{thm}

	\begin{lem}\label{local-lemma}
		Let $X$ be a qcqs scheme. Then the stack $\Equiv_{\cQCoh(X)_{\geq 0}}(\cQCoh(X)_{\geq 0},\cQCoh(X)_{\geq 0})$ is equivalent to $\BGm\times X$.
	\end{lem}
	\begin{proof}
		Let $S\to X$ be a morphism of schemes. A $\QCoh(S)$-linear autoequivalence of $\QCoh(S)$ is determined by the image of $\cO_S$, and therefore amounts to the datum of a line bundle $\cL$ concentrated in degree $0$ (the degree must be nonnegative because we are in the connective setting, and if it were strictly positive then the inverse functor would be given by tensoring by a negatively-graded line bundle, which is impossible).
		
		This implies that the desired moduli space is $\cPic\times X$, which is the same as $\BGm\times X$.
	\end{proof}
	
	\begin{prop}\label{triv-gerbe}Let $X$ be a qcqs scheme, and $M\in \cBr(X)$. Then $\Triv_{\geq 0}(M)$ has a natural structure of $\Gm$-gerbe over $X$.
	Furthermore, given an isomorphism $f:M \rightarrow N$ of prestable invertible categories, then the morphism $\Triv_{\geq 0}(f)$ defined by the association $\phi\mapsto f \circ \phi$ is a morphism of $\Gm$-gerbes.\end{prop}
	\begin{proof}
		
		To prove that $\Triv_{\geq 0}(M)$ is a $\Gm$-gerbe we need to verify that $\Triv_{\geq 0}(M)$ is locally nonempty and connected (i.e. locally in the base, it is nonempty and any two pair of objects are connected by an isomorphism) and to provide a $\Gm$-banding. The first two assertions are evident from the fact that $\Triv_{\geq 0}(M)$ is locally (in the base) of the form $X \times \BGm$ by \cref{local-lemma}.
		Now we provide the banding in the following way. First of all, notice that $\mathbb G_{\textup{m},X}$ can be identified with the automorphism group of the identity endofunctor of $\QCoh(X)_{\geq 0}$, namely $\cO_X$-linear invertible natural transformations of $\id_{\QCoh(X)_{\geq 0}}$. Let $\cI_{\Triv_{\geq 0}(M)}$ be the inertia stack of $\Triv_{\geq 0}(M)$. We define a functor 
		$$ \alpha_M: \Triv_{\geq 0}(M) \times \Gm \longrightarrow \cI_{\Triv_{\geq 0}(M)} $$ 
		as $\alpha_M(\phi,\lambda):= (\phi, \id_{\phi} * \lambda) \in \cI_{\Triv_{\geq 0}(M)}$ for every $(\phi,\lambda)$ object of $\Triv_{\geq 0}(M)$, $*$ being the horizontal composition of natural transformations. Since being an isomorphism is an \'etale-local property, we can reduce to the trivial case, for which both sides of the banding are isomorphic to $(X \times \B\Gm) \times \Gm$ while the morphism is just the identity. The second part of the statement follows from a straightforward computation. More precisely, we need to prove that the following diagram 
        $$
        \begin{tikzcd}
        \Triv_{\geq 0}(M) \times \Gm \arrow[d, "\Triv_{\geq 0}(f)"] \arrow[r, "\alpha_M"] & \cI_{\Triv_{\geq 0}(M)} \arrow[d, "\cI_{\Triv_{\geq 0}(f)}"] \\
        \Triv_{\geq 0}(N) \times \Gm \arrow[r, "\alpha_N"]                             & \cI_{\Triv_{\geq 0}(N)}                   
        \end{tikzcd}
        $$
        is commutative, where $\cI_{\Triv_{\geq 0}(f)}$ is the morphism induced by $\Triv_{\geq 0}(f)$ on the inertia stacks. This follows from the trivial equality $\id_{f} * \id_{\phi} = \id_{f \circ \phi}$.
     
	\end{proof}
	
\begin{rem}
     The following lemma is an example of a result which is more natural to prove in the context of $\infty$-categories, or in particular $\infty$-operads. In the proof of \cref{id-is-invertible}, we essentially proved that there exists an isomorphism 
     $$ \QCoh_{\id}(G \star G') \simeq \QCoh_{\id}(G) \otimes_{\QCoh(X)} \QCoh_{\id}(G')$$
     and then proved that it is compatible with the higher structures. This follows essentially from the close relation between the $\infty$-categories of twisted sheaves on $G\times_X G'$ and on $G \star G'$ (see \cref{rigidification-pullback-theorem}). 

     As far as the functor $\Triv_{\geq 0}(-)$ is concerned, the same cannot be done so easily. One would like to use the morphism of gerbes
     $$\Triv_{\geq 0}(M) \times_X \Triv_{\geq 0}(N) \rightarrow \Triv_{\geq 0}(M) \star \Triv_{\geq 0}(N)$$
     but there is no natural way to construct a morphism 
     $$ \Triv_{\geq 0}(M \otimes N) \rightarrow \Triv_{\geq 0}(M) \times_X \Triv_{\geq 0}(N)$$ 
     as in the proof of \cref{id-is-invertible}. However, the two symmetric monoidal structures are constructed using the (cartesian) product structure. Therefore, the idea is to prove that the $\otimes$-structure naturally transforms into the $\star$-structure through $\Triv_{\geq 0}$ as $\infty$-operad. It is important to remark that $\Triv_{\geq 0}(\prod M_i)$ is not isomorphic to $\prod \Triv_{\geq 0}( M_i)$. Nevertheless, we just need to universal property of the product to prove \cref{symm-mon}.
     \end{rem}

 \begin{lem}\label{symm-mon}
	    The functor $\Triv_{\geq 0}(-)$ is symmetric monoidal.
	\end{lem}	
 \begin{proof}
	To upgrade $\Triv_{\geq 0}$ to a symmetric monoidal functor, we need to define its action on multilinear maps. Let $\{ M_i \}_{i \in I}$ be a sequence of invertible prestable $\cO_X$-linear $\infty$-categories indexed by a finite set $I$ and $N$ be another invertible prestable $\cO_X$-linear category. Then we define 
	$$ \Triv_{\geq 0}(-): \textup{Mul}\Big(\{ M_i\}_{i \in I}, N\Big) \longrightarrow \textup{Mul}\Big(\{ \Triv_{\geq 0}(M_i)\}_{i \in I}, \Triv_{\geq 0}(N) \Big)$$
	in the following way: if $f:\prod_{i \in I} M_i \rightarrow N$ is a morphism in $\Linpre(X)$ which preserves small colimits separately in each variable, we have to define the image $\Triv_{\geq 0}(f)$ as a functor 
	$$\Triv_{\geq 0}(f): \prod_{i \in I} \Triv_{\geq 0}(M_i) \longrightarrow \Triv_{\geq 0}(N)$$ 
	such that $\textup{Band}(\Triv_{\geq 0}(f))$ is the $n$-fold multiplication of $\Gm$, where $n$ is the cardinality of $I$. We define it on objects in the following way: if $\{ \phi_i \}$ is an object of $\prod_{i \in I}(\Triv_{\geq 0}(M_i))$, then we set 
	$$\Triv_{\geq 0}(f) \Big(\{ \phi_i \} \Big) := f \circ \prod \phi_i$$
	where $\prod \phi_i$ is just the morphism induced by the universal property of the product. 
	
	First of all we need to prove that $f \circ \prod \phi_i$ is still an equivalence. Let $S \in \Sch_X$ and $f: \prod_{i\in I}M_i \rightarrow N$ morphism in $\Linpre(X)$, we can consider the following diagram:
	\begin{equation}\label{diagram}
	\begin{tikzcd}
\QCoh(S)  \arrow[r, "\prod \phi_i"] \arrow[d, "\otimes \phi_i"'] & \prod_{i \in I} \cM_i(S) \arrow[ld,"u"'] \arrow[d, "f(S)"] \\
\bigotimes_{i\in I} \cM_i(S) \arrow[r, "\tilde{f}(S)"]                   & \mathcal{N}(S)
\end{tikzcd}
\end{equation}
where the tensors are in fact relative tensors over $\QCoh(S)$. The diagonal map $u$ is the morphism universal between all the $\QCoh(X)$-linear morphisms from $\prod_{i \in I} \cM_i(S)$ which preserve small colimits separately in each variable. Equivalently, one can say that it is a coCartesian edge in the $\infty$-operad $\Linpre(S)^{\otimes}$.

Thus, it is enough to prove that both $\otimes \phi_i$ and $\tilde{f}$ are equivalences. Because the source of the functor $\Triv_{\geq 0}(-)$ is the $\infty$-groupoid $\cBr(S)$, the morphism $\tilde{f}(S)$ is an equivalence. Furthermore, the morphisms $\phi_i$ are equivalences, therefore it follows from the functoriality of the relative tensor product that $\otimes \phi_i$ is an equivalence. A straightforward computation shows that it is defined also at the level of $1$-morphisms and it is in fact a functor.

It remains to prove that $\textup{Band}\Big(\Triv_{\geq 0}(f)\Big)$ is the $n$-fold multiplication of $\Gm$, where $n$ is the cardinality of $I$. It is equivalent to prove the commutativity of the following diagram:
$$\begin{tikzcd}
\prod_{i \in I}\Triv_{\geq 0}(M_i) \times \Gm^n \arrow[rr, "{(\Triv_{\geq 0}(f), m^n)}"] \arrow[d, "\prod \alpha_{M_i}"] &  & \Triv_{\geq 0}(N) \times \Gm \arrow[d, "\alpha_N"] \\
\prod_{i \in I} \mathcal{I}_{\Triv_{\geq 0}(M_i)} \arrow[rr, "\mathcal{I}_{\Triv_{\geq 0}(f)}"]                         &  & \mathcal{I}_{\Triv_{\geq 0}(N)}
\end{tikzcd}$$
where $m^n$	is the $n$-fold multiplication of $\Gm$. The notation follows the one in the proof of \cref{triv-gerbe}. Using diagram \ref{diagram} again, we can reduce to the following straightforward statement: let $\lambda_1,\dots,\lambda_n$ be $\QCoh(S)$-linear automorphisms of $\id_{\QCoh(S)}$, which can be identified with elements of $\Gm(S)$; then the tensor of the natural transformations coincide with the product as elements of $\Gm$, i.e $\lambda_1 \otimes \dots \otimes \lambda_n = m^n(\lambda_1\dots\lambda_n)$. 

Finally, because all maps are maps of gerbes, the condition of being strictly monoidal is automatically satisfied once the lax monoidal structure is given.
\end{proof}
	
	\subsection{Proof of the main theorem}\label{section-equivalence}
	The goal of this subsection is to prove that the constructions $$M \mapsto \Triv_{\geq 0}(M)$$
	and
	$$G\mapsto \QCoh_{\id}(G)_{\geq 0}$$
	establish a (symmetric monoidal) categorical equivalence between the $\infty$-groupoids $\cBr(X)$ and $\Ger_{\Gm}(X)$, thus proving \cref{main-theorem}.
 \begin{rem}\label{2-groupoids}Note that both sides of \cref{main-theorem} are in fact 2-groupoids, the left-hand-side by \cite[Construction 11.5.7.13]{SAG} and the right-hand-side by definition. However, we do not use this in the proof.\end{rem}

	\begin{prop}\label{Triv-is-ff}The functor $\Triv_{\geq 0}(-)$ is fully faithful.\end{prop}
	
	\begin{proof}We want to prove that for any $M,M'\in \cBr(X)$ the map $$\Map_{\cBr(X)}(M,M')\to \Map_{\Ger_{\Gm}(X)}(\Triv_{\geq 0}(M),\Triv_{\geq 0}(M'))$$ is a homotopy equivalence of $\infty$-groupoids.
 First of all, by using the grouplike monoid structure of $\cBr(X)$ together with \cref{symm-mon}, we can reduce to the case $M'=\QCoh(X)_{\geq 0}=\mathbf 1$.
		
		If $M$ does not lie in the connected component of $\mathbf 1$, then what we have to check is that   $$\Map_{\Ger_{\Gm}(X)}(\Triv_{\geq 0}(M),\Triv_{\geq 0}(\mathbf 1))=\varnothing.$$ But if this space contained an object, then in particular we would have an equivalence at the level of global sections between $\textup{Equiv}_{\QCoh(X)_{\geq 0}}(M,\mathbf 1)$ and $\textup{Equiv}_{\QCoh(X)_{\geq 0}}(\mathbf 1,\mathbf 1)$. But the first space is empty by hypothesis, while the second is not.

 On the other hand, if $M$ lies in the connected component of $\mathbf 1$, by functoriality of $\Triv_{\geq 0}(-)$ we can suppose that $M=\mathbf 1$. In this case, we have to prove that the map $$\Map_{\cBr(X)}(\mathbf 1, \mathbf 1)\to \Map_{\Ger_{\Gm}(X)}(\Triv_{\geq 0}(\mathbf 1),\Triv_{\geq 0}(\mathbf 1))\simeq \Map_{\Ger_{\Gm}(X)}(X\times \BGm, X\times \BGm)$$ is an equivalence. But the first space is the groupoid $\Pic(X)$, the latter is the space of maps $X\to \BGm$, and the composite map is the one sending a line bundle over $X$ to the map $X\to \BGm$ that classifies it.
	\end{proof}
	
	\begin{prop}\label{ess-surj} Let $X$ be a qcqs scheme. Then there is a natural equivalence of functors $$\textup{Id}_{\Ger_{\Gm}(X)}\Rightarrow\Triv_{\geq 0}\circ \QCoh_\id(-)_{\geq 0}.$$
	\end{prop}
	\begin{proof}
		Let $G$ be a $\Gm$-gerbe over $X$. Let us observe that for any $S\to X$ we have $$G(S)\simeq \Equiv_{\Ger_{\Gm}(S)}(S\times \BGm,G_S).$$
		
		Indeed, there is a map of stacks over $X$
		
		$$ F:\Equiv_{\Ger_{\Gm}}(X\times \BGm, G)\to G$$ 
		$$(S\to X,\phi_S:S\times \BGm\to G_S)\mapsto (S\to X,\phi\circ u_S)$$
  where $u_S:S\to  S\times \BGm$ is induced by the canonical atlas of $\BGm$. Up to passing to a suitable \'etale covering of $X$, the map becomes an equivalence: in fact the choice of an equivalence $\phi:S\times \BGm\to S\times\BGm$ \textit{of gerbes over $S$} amounts to the choice of a map $S\to \BGm$, because $\phi$ must be a map over $S$ (hence $\pr_{\BGm}\circ\phi=\pr_{\BGm}$) and it must respect the banding. Therefore, $F$ is an equivalence over $X$, and this endows $\Equiv_{\Ger_{\Gm}}(X\times \BGm,G)$ with a natural structure of a $\Gm$-gerbe over $X$.
		The construction $\phi\mapsto \phi_*$ thus provides a morphism of stacks $$G\to \Equiv_{\cQCoh(X)}\Big(\cQCoh_\id(X\times \BGm),\cQCoh_\id(G)\Big).$$ 
		 
		Now since the pushforward $\phi_*$ of an equivalence $\phi$ of stacks is t-exact, the construction $\phi\mapsto \phi_*|_{\cQCoh_{\id}(X\times \BGm)_{\geq 0}}$ yields a map $$G\to \Equiv_{\cQCoh(X)_{\geq 0}}\Big(\cQCoh_\id(X\times \BGm)_{\geq 0},\cQCoh_\id(G)_{\geq 0}\Big).$$ and now the right-hand-side is in turn equivalent to
		$$G\to \Equiv_{\cQCoh(X)_{\geq 0}}\Big(\cQCoh(X)_{\geq 0},\cQCoh_\id(G)_{\geq 0} \Big)=\Triv_{\geq 0}(\QCoh_\id(G)_{\geq 0}).$$
		
		But now, $\Triv_{\geq 0}(\QCoh_\id(G))$ is a $\Gm$-gerbe over $X$ (note that this was not true before passing to the connective setting), and therefore to prove that the map is an equivalence it suffices to prove that it agrees with the bandings, i.e. that it is a map of $\Gm$-gerbes. This follows from unwinding the definitions.
\end{proof}

	We are now ready to prove our main result.
	
	\
	
	\textit{Proof of \cref{main-theorem}}: \cref{QCoh-id-cn-is-invertible} and \cref{Triv-is-gerbe}  tell us that the two functors are symmetric monoidal and take values in the sought $\infty$-categories. The fact that they form an equivalence follows from \cref{Triv-is-ff} and \cref{ess-surj}.\qed

	\bibliographystyle{alpha}
	\bibliography{Azumaya.bib}
	
\end{document}